\documentclass[11pt]{article}
\usepackage[T1]{fontenc}
\usepackage{amsmath,amsthm,amsfonts,amssymb,microtype,thmtools,underscore,mathtools,anyfontsize,thm-restate,graphicx,mathtools}
\usepackage[usenames,dvipsnames,svgnames,table]{xcolor}
\usepackage[unicode=true]{hyperref}
\hypersetup{
colorlinks=true,
breaklinks=true,
linkcolor={black},
citecolor={black},
urlcolor={blue!60!black},
pdftitle={Product Structure of Graph Classes with Strongly Sublinear Separators}
pdfauthor ={Zden\v{e}k Dvo\v{r}\'{a}k, David R. Wood}}
\usepackage[noabbrev,capitalise,nameinlink]{cleveref}
\crefname{lem}{Lemma}{Lemmas}
\crefname{thm}{Theorem}{Theorems}
\crefname{cor}{Corollary}{Corollaries}
\crefname{prop}{Proposition}{Propositions}
\crefname{conj}{Conjecture}{Conjectures}
\crefname{open}{Open Problem}{Open Problems}
\crefformat{equation}{(#2#1#3)}
\Crefformat{equation}{Equation #2(#1)#3}
\crefformat{enumi}{#2#1#3}
\Crefformat{enumi}{Item (#2#1#3)}
\usepackage[shortlabels]{enumitem}
\setlist[itemize]{topsep=0ex,itemsep=0ex,parsep=0.25ex}
\setlist[enumerate]{topsep=0ex,itemsep=0ex,parsep=0.25ex}
\newcommand{\dims}{d}
\newcommand{\tdb}{t}
\newcommand{\defn}[1]{\textcolor{Maroon}{\emph{#1}}}

\newcommand{\WW}{\mathcal{W}}

\newcommand{\GG}{\mathcal{G}}

\newcommand{\PP}{\mathcal{P}}

\newcommand{\NN}{\mathbb{N}}
\newcommand{\REALS}{\mathbb{R}}
\usepackage[longnamesfirst,numbers,sort&compress]{natbib}
\newcommand{\mycitet}[1]{\mbox{\citet{#1}}}
\usepackage[bmargin=25mm,tmargin=25mm,lmargin=30mm,rmargin=30mm]{geometry}
\renewcommand{\baselinestretch}{1.1}
\allowdisplaybreaks

\DeclarePairedDelimiter{\abs}{\lvert}{\rvert}
\DeclarePairedDelimiter{\floor}{\lfloor}{\rfloor}
\DeclarePairedDelimiter{\ceil}{\lceil}{\rceil}
\newcommand{\bigfloor}[1]{\big\lfloor{#1}\big\rfloor}
\renewcommand{\ge}{\geqslant}
\renewcommand{\le}{\leqslant}
\renewcommand{\geq}{\geqslant}
\renewcommand{\leq}{\leqslant}
\renewcommand{\emptyset}{\varnothing}
\DeclareMathOperator{\dist}{dist}

\DeclareMathOperator{\sep}{sep}

\DeclareMathOperator{\tw}{tw}
\DeclareMathOperator{\pw}{pw}

\DeclareMathOperator{\spw}{spw}

\DeclareMathOperator{\tpw}{tpw}
\DeclareMathOperator{\td}{td}

\renewcommand{\thefootnote}{\fnsymbol{footnote}}
\allowdisplaybreaks
\theoremstyle{plain}
\newtheorem{thm}{Theorem}
\newtheorem{lem}[thm]{Lemma}
\newtheorem{open}[thm]{Open Problem}
\newtheorem*{claim}{Claim}
\newtheorem{cor}[thm]{Corollary}
\newtheorem{prop}[thm]{Proposition}
\newtheorem{obs}[thm]{Observation}
\theoremstyle{definition}

\newcommand{\omitted}[1]{

\smallskip
\framebox{\parbox{\textwidth}{#1}}

\smallskip}
\renewcommand{\omitted}[1]{}

\begin{document}

\footnotetext{\today}

\author{Zden\v{e}k Dvo\v{r}\'{a}k\footnotemark[2] \qquad 
David~R.~Wood\footnotemark[2]}

\footnotetext[1]{Institute of Computer Science, Charles University, Prague, Czech Republic (\texttt{rakdver@iuuk.mff.cuni.cz}). Supported by project 22-17398S (Flows and cycles in graphs on surfaces) of Czech Science Foundation.}

\footnotetext[2]{School of Mathematics, Monash University, Melbourne, Australia  (\texttt{david.wood@monash.edu}). Research supported by the Australian Research Council and a Visiting Research  Fellowship of Merton College, University of Oxford. }

\sloppy
	
\title{\Large\textbf{Product Structure of Graph Classes\\ with Strongly Sublinear Separators}}

\date{}	

\maketitle

\begin{abstract}
We investigate the product structure of hereditary graph classes admitting strongly sublinear separators. 
We characterise such classes as subgraphs of the strong product of a star and a complete graph of strongly sublinear size.
In a more precise result, we show that if any hereditary graph class $\GG$ admits $O(n^{1-\epsilon})$ separators, then for any fixed $\delta\in(0,\epsilon)$ every $n$-vertex graph in $\GG$ is a subgraph of the strong product of a graph $H$ with bounded tree-depth and a complete graph of size $O(n^{1-\epsilon+\delta})$. This result holds with $\delta=0$ if we allow $H$ to have tree-depth $O(\log\log n)$. Moreover, using extensions of classical isoperimetric inequalties for grids graphs, we show the dependence on $\delta$ in our results and the above $\td(H)\in O(\log\log n)$ bound are both best possible. We prove that $n$-vertex graphs of bounded treewidth are subgraphs of the product of a graph with tree-depth $\tdb$ and a complete graph of size $O(n^{1/\tdb})$, which is best possible. Finally, we investigate the conjecture that for any hereditary graph class $\GG$ that admits $O(n^{1-\epsilon})$ separators, every $n$-vertex graph in $\GG$ is a subgraph of the strong product of a graph $H$ with bounded tree-width and a complete graph of size $O(n^{1-\epsilon})$. We prove this for various classes $\GG$ of interest. 
\end{abstract}


\renewcommand{\thefootnote}{\arabic{footnote}}

\newpage
\section{Introduction}
\label{Introduction}

Graph\footnote{We consider simple, finite, undirected graphs~$G$ with vertex-set~${V(G)}$ and edge-set~${E(G)}$. A \defn{graph class} is a collection of graphs closed under isomorphism. A graph class is \defn{hereditary} if it is closed under taking induced subgraphs. A graph class is \defn{monotone} if it is closed under taking subgraphs. A graph~$H$ is \defn{contained} in a graph $G$ if $H$ is isomorphic to a subgraph of $G$. A graph~$H$ is a \defn{minor} of a graph~$G$ if~$H$ is isomorphic to a graph obtained from a subgraph of~$G$ by contracting edges. A graph~$G$ is \defn{$H$-minor-free} if~$H$ is not a minor of~$G$. A graph class~$\GG$ is \defn{minor-closed} if every minor of each graph in~$\GG$ is also in~$\GG$. A \defn{$K_h$-model} in a graph $G$ consists of pairwise-disjoint vertex-sets $(U_1, \dotsc, U_h)$ such that, $G[U_i]$ is connected for each $i$ and 
$G[U_i\cup U_j]$ is connected for all $i,j$. Clearly $K_h$ is a minor of a graph $G$ if and only if $G$ contains a $K_h$-model. See \citep{Diestel5} for graph-theoretic definitions not given here. The \defn{Euler genus} of a surface with $h$ handles and $c$ cross-caps is $2h + c$. The \defn{Euler genus} of a graph $G$ is the minimum integer $g\geq 0$ such that $G$ embeds in a surface of Euler genus $g$; see \cite{MoharThom} for more about graph embeddings in surfaces.} product structure theory describes complicated graphs as subgraphs of strong products\footnote{The \defn{strong product} of graphs~$A$ and~$B$, denoted by~${A \boxtimes B}$, is the graph with vertex-set~${V(A) \times V(B)}$, where distinct vertices ${(v,x),(w,y) \in V(A) \times V(B)}$ are adjacent if
	${v=w}$ and ${xy \in E(B)}$, or
	${x=y}$ and ${vw \in E(A)}$, or
	${vw \in E(A)}$ and~${xy \in E(B)}$.
} of simpler building blocks. Examples of graphs classes that can be described this way include planar graphs \citep{DJMMUW20,UWY22}, graphs of bounded Euler genus~\citep{DJMMUW20,DHHW22}, graphs excluding a fixed minor~\citep{DJMMUW20,ISW,UTW}, various non-minor-closed classes~\citep{DMW23,HW21b,BDHK22,DHSW}, and graphs of bounded tree-width~\citep{UTW,DO95}. These results have been the key to solving several long-standing open problems 
about queue layouts~\citep{DJMMUW20}, nonrepetitive colourings~\citep{DEJWW20}, centred colourings~\citep{DFMS21}, adjacency labelling~\cite{EJM,DEJGMM21}, twin-width~\cite{BKW,KPS23,JP22}, vertex ranking~\citep{BDJM}, and box dimension~\citep{DGLTU22}. This paper studies the product structure of graph classes with strongly sublinear separators, which is a more general setting than all of the above classes. 

\subsection{Background}

A \defn{balanced separator} in an $n$-vertex graph $G$ is a set $S\subseteq V(G)$ such that each component of $G-S$ has at most $\frac{n}{2}$ vertices. The \defn{separation-number $\sep(G)$} of a graph $G$ is the minimum size of a balanced separator in $G$.  

For a graph parameter $f$ and graph class $\GG$, let $f(\GG)$ be the function $n\mapsto \max\{f(G):G\in\GG,\,|V(G)|=n\}$. 
We say $\GG$ has \defn{strongly sublinear} $f$ if $f(\GG)\in O(n^{1-\epsilon})$ for some fixed $\epsilon>0$ (where $n$ is always the number of vertices).

Many classes of graphs have strongly sublinear separation-number. For example, \citet{LT79} proved that planar graphs have separation-number $O(n^{1/2})$. More generally, \citet{Djidjev81} and \citet{GHT-JAlg84} independently proved that graphs embeddable on any fixed surface have separation-number $O(n^{1/2})$. More generally still, \citet{AST90} proved that any proper minor-closed class has separation-number $O(n^{1/2})$. Many non-minor-closed classes also have strongly sublinear separation-number. For example, \citet{GB07} proved that graphs that have a drawing in the plane with a bounded number of crossings per edge have separation-number $O(n^{1/2})$. And \citet{MTTV97} proved that touching graphs of $d$-dimensional spheres have separation-number $O(n^{1-1/d})$ (amongst more general results). 

The following characterisation of graph classes with strongly sublinear separation-number in terms of tree-width\footnote{For a tree~$T$, a \defn{$T$-decomposition} of a graph~$G$ is a collection~${\WW = (W_x \colon x \in V(T))}$ of subsets of~${V(G)}$ indexed by the nodes of~$T$ such that
(i) for every edge~${vw \in E(G)}$, there exists a node~${x \in V(T)}$ with~${v,w \in W_x}$; and 
(ii) for every vertex~${v \in V(G)}$, the set~${\{ x \in V(T) \colon v \in W_x \}}$ induces a (connected) subtree of~$T$. 
Each set~$W_x$ in~$\WW$ is called a \defn{bag}. 
The \defn{width} of~$\WW$ is~${\max\{ \abs{W_x} \colon x \in V(T) \}-1}$. 
A \defn{tree-decomposition} is a $T$-decomposition for any tree~$T$. 
The \defn{tree-width~$\tw(G)$} of a graph~$G$ is the minimum width of a tree-decomposition of~$G$. Tree-width is the standard measure of how similar a graph is to a tree. Indeed, a connected graph has tree-width 1 if and only if it is a tree. 
A \defn{path-decomposition} is a $P$-decomposition for any path~$P$. 
The \defn{path-width $\pw(G)$} of a graph~$G$ is the minimum width of a path-decomposition of $G$. }, path-width and tree-depth\footnote{A forest is \defn{rooted} if each component has a nominated \defn{root} vertex (which defines the ancestor relation). The \defn{vertex-height} of a rooted forest is the maximum number of vertices in a root--leaf path. The \defn{closure} of a rooted forest $F$ is the graph $G$ with $V(G)\coloneqq V(F)$  with $vw\in E(G)$ if and only if $v$ is an ancestor of $w$ or vice versa. The \defn{tree-depth $\td(G)$} of a graph $G$ is the minimum vertex-height of a rooted forest $T$ such that $G$ is a subgraph of the closure of $T$.}  is folklore. 

\begin{thm}
\label{Folklore}
For fixed $\epsilon\in(0,1)$, the following are equivalent for any hereditary graph class $\mathcal{G}$:
\begin{enumerate}[\textnormal{(\alph{*})}]
\item\label{SS} $\GG$ has separation-number $\sep(\GG)\in O(n^{1-\epsilon})$,
\item\label{SStw} $\GG$ has tree-width $\tw(\GG)\in O(n^{1-\epsilon})$,
\item\label{SSpw} $\GG$ has path-width $\pw(\GG)\in O(n^{1-\epsilon})$,
\item\label{SStd} $\GG$ has tree-depth $\td(\GG)\in O(n^{1-\epsilon})$.
\end{enumerate}
\end{thm}

\begin{proof}
It follows from the definitions that for every graph $G$, 
$$\tw(G)\leq \pw(G)\leq \td(G)-1.$$ 
Thus \cref{SStd} $\Rightarrow$ \cref{SSpw} $\Rightarrow$ \cref{SStw}. \citet[(2.6)]{RS-II} showed that for every graph $G$,
\begin{equation}
    \label{septw}
    \sep(G)\leq\tw(G)+1.
\end{equation}    
Thus \cref{SStw} $\Rightarrow$ \cref{SS}. It is folklore that \cref{SS} $\Rightarrow$ \cref{SStd}; see \citep{Bodlaender98,AFHJW21} for proofs that \cref{SS} $\Rightarrow$ \cref{SSpw} which are easily adapted to show that \cref{SS} $\Rightarrow$ \cref{SStd}. (Note that \citet{DN19} proved a stronger relationship between balanced separators and tree-width that holds without the `strongly sublinear' assumption.)
\end{proof}

We only consider products of the form $H \boxtimes K_m$, which is the `complete blow-up' of the graph $H$, obtained by replacing each vertex of $H$ by a copy of the complete graph $K_m$ and each edge of $H$ by a copy of the complete bipartite graph $K_{m,m}$. Such products can be characterised via the following definition. For graphs $H$ and $G$, an \defn{$H$-partition} of $G$ is a partition  $\PP=(V_x:x\in V(H))$ of $V(G)$ indexed by the vertices of $H$ such that for each edge $vw\in E(G)$, if $v\in V_x$ and $w\in V_y$ then $xy\in E(H)$ or $x=y$. 
The \defn{width} of such an $H$-partition is 
$\max\{|V_x|:x\in V(H)\}$. The following observation follows immediately from the definitions. 

\begin{obs}
\label{ProductPartition}
For any graph $H$ and $m \in\mathbb{N}$, a graph $G$ is contained in $H \boxtimes K_m$ if and only if $G$ has an $H$-partition of width $m$. 
\end{obs}

If $\PP$ is an $H$-partition of a graph $G$, where $H$ is a tree or a star, then $\PP$ is respectively called a \defn{tree-} or a \defn{star-partition}.\footnote{Tree-partitions were independently introduced by \mycitet{Seese85} and \citet{Halin91}, and have since been widely investigated \citep{Bodlaender-DMTCS99,BodEng-JAlg97,DO95,DO96,Edenbrandt86,Wood06,Wood09,UTW,BGJ22,DW22}. Tree-partition-width has also been called \defn{strong tree-width}~\citep{BodEng-JAlg97, Seese85}. Applications of tree-partitions include 
graph drawing~\citep{CDMW08,GLM05,DMW05,DSW07,WT07}, nonrepetitive graph colouring~\citep{BW08}, 
clustered graph colouring~\citep{ADOV03}, 
monadic second-order logic~\citep{KuskeLohrey05}, 
size Ramsey numbers~\citep{DKCPS}, 
network emulations~\citep{Bodlaender-IPL88, Bodlaender-IC90, BvL-IC86, FF82}, 
statistical learning theory~\citep{ZA22}, and 
the edge-{E}rd{\H{o}}s-{P}{\'o}sa property~\citep {RT17,GKRT16,CRST18}. 
Planar-partitions and other more general structures have also been studied~\citep{DK05,RS98,WT07,DOSV00,DOSV98}.}
The \defn{tree-partition-width $\tpw(G)$} is the minimum width of a tree-partition of $G$, and the \defn{star-partition-width $\spw(G)$} is the minimum width of a star-partition of $G$. By \cref{ProductPartition}, $\tpw(G)\leq k$ if and only if $G$ is contained in $T\boxtimes K_k$ for some tree $T$, and $\spw(G)\leq k$ if and only if $G$ is contained in $S\boxtimes K_k$ for some star $S$. 

\subsection{Bounded Width Partitions}

Our first result (proved in \cref{StarPartitions}) characterises graph classes with strongly sublinear separation-number in terms of star- and tree-partition-width. 

\begin{restatable}{thm}{Primary}
\label{Primary}
The following are equivalent for any hereditary graph class $\GG$:
\begin{enumerate}[\textnormal{(\roman{*})}]
\item \label{SSS} $\GG$ has strongly sublinear separation-number, 
\item\label{SStpw} $\GG$ has strongly sublinear tree-partition-width,
\item\label{SSspw} $\GG$ has strongly sublinear star-partition-width.
\end{enumerate}
\end{restatable}

We emphasise that $\epsilon$ is fixed in \cref{Folklore}, but not in \cref{Primary}. Indeed, there can be a significant difference between the exponents in the bounds on the separator-number and the tree- or star-partition-width in \cref{Primary}. Call this the `\defn{exponent-gap}'. For example, while planar graphs have separation-number
$\Theta(n^{1/2})$, in \cref{StarPartitions} we show that $\Theta(n^{2/3})$ is a tight bound on the star-partition-width and the tree-partition-width of planar graphs. In this case, the exponent-gap is $\frac23 - \frac12 = \frac16$.

This motivates the question: can the exponent-gap be reduced by considering products $H\boxtimes K_m$ for graphs $H$ that are more complicated than stars or trees, but still with bounded tree-width or bounded tree-depth? In particular, can the exponent-gap be 0? 

\citet{ISW} achieved an exponent-gap of 0 for minor-closed graph classes where $H$ has bounded tree-width:

\begin{thm}[\citep{ISW}]
\label{ISW}\hspace*{5mm}
\begin{enumerate}[\textnormal{(\alph{*})}]
\item Every $n$-vertex planar graph is contained in $H\boxtimes K_m$, for some graph $H$ with $\tw(H)\leq 3$, where $m \leq \sqrt{8n}$.
\item Every $n$-vertex graph of Euler genus $g$ is contained in  $H \boxtimes K_m$, for some graph $H$ with  $\tw(H) \leq 3$, where $m \leq 4\sqrt{(g+1)n}$.
\item Every $n$-vertex $K_{s,t}$-minor-free graph is contained in $H\boxtimes K_m$ for some graph $H$ with  $\tw(H)\leq s$, where  $m\leq \sqrt{(s-1)(t-1)n}$.
\item Every $n$-vertex $K_t$-minor-free graph is contained in  $H\boxtimes K_m$ for some graph $H$ with $\tw(H)\leq t-2$, where  $m\leq 2\sqrt{(t-3)n}$.
\end{enumerate}
\end{thm}

In all these results the dependence on $n$ is best possible because $\tw(H\boxtimes K_m)\leq(\tw(H)+1)m-1$ and the $n^{1/2}\times n^{1/2}$ grid has tree-width $n^{1/2}$. 

Can similar results be obtained for an arbitrary graph class with strongly sublinear separation-number? The short answer is `almost', as expressed in the next theorem which shows that the exponent-gap can be made arbitrarily small. 

\begin{restatable}{thm}{Almost}
\label{Almost}
For $\delta,\epsilon\in\REALS$ with $0<\delta<\epsilon<1$ there exists $\tdb\in\NN$ with the following property. Let $\GG$ be any hereditary graph class with  $\sep(\GG)\leq cn^{1-\epsilon}$ for some constant $c>0$. Then every $n$-vertex graph $G\in\GG$ is contained in $H \boxtimes K_m$ for some graph $H$ with $\td(H)\leq\tdb$, where $m \leq \frac{c\, 2^\epsilon}{2^\epsilon-1} \, n^{1-\epsilon+\delta}$.
\end{restatable}

In \cref{Almost} the graph $H$ has bounded tree-depth. This setting generalises bounded star-partition-width since a graph has tree-depth 2 if and only if it is a star (plus isolated vertices). 

\cref{Almost} is proved in \cref{BoundedTreedepth}, where we also prove a result (\cref{AlmostAlmost}) in which the exponent-gap tends to 0 as $n\to\infty$, at the expense of allowing $\td(H)$ to increase (arbitrarily slowly) with $n$. Another result (\cref{AlmostAlmostAlmost}) has exponent gap 0 and $\td(H)\in O(\log\log n)$. 

In \cref{LowerBounds} we prove various lower bounds that show that many of the results in \cref{BoundedTreedepth} are best possible for multi-dimensional grid graphs. In particular, we show that the dependence on $\delta$ in \cref{Almost} is best possible (\cref{DimGridPartitionTreeDepthCor}), and that the above $\td(H)\in O(\log\log n)$ bound is best possible (\cref{DimAlmostAlmostAlmostTight}). 

It is open whether exponent gap 0 and $\tw(H)\in O(1)$ can be  achieved simultaneously.

\begin{open}
\label{mainquestion}
For any hereditary graph class $\GG$ with separation-number $\sep(\GG)\in O(n^{1-\epsilon})$, does there exist a constant $c=c(\GG)$ such that every $n$-vertex graph $G\in \GG$ is contained in $H\boxtimes K_m$, where $\tw(H)\leq c$ and $m\in O(n^{1-\epsilon})$.
\end{open}

\cref{ISW} solves \cref{mainquestion} for minor-closed classes with $\epsilon=\frac12$. It may even be true that $c$ in \cref{mainquestion} is a function of $\epsilon$ only, which was recently proved for minor-closed classes by \citet{DDEHJMMSW}, who also gave improved tree-width bounds for $K_{3,t}$-minor-free graphs, which includes planar and graphs with bounded Euler genus. 

\begin{thm}[\citep{DDEHJMMSW}]
\hspace*{5mm}
\begin{enumerate}[\textnormal{(\alph{*})}]
\item Every $n$-vertex planar graph is contained in  $H\boxtimes K_m$, where $\tw(H)\leq 2$ and $m\in O(n^{1/2})$.
\item Every $n$-vertex graph of Euler genus $g$ is contained in  $H\boxtimes K_m$, where $\tw(H)\leq 2$ and $m\in O(gn^{1/2})$.
\item Every $n$-vertex $K_{3,t}$-minor-free graph is contained in  $H\boxtimes K_m$, where $\tw(H)\leq 2$ and $m\in O(tn^{1/2})$.
\item Every $n$-vertex $K_t$-minor-free graph is contained in  $H\boxtimes K_m$, where $\tw(H)\leq 4$ and $m\in O_t(n^{1/2})$.
\end{enumerate}
\end{thm}

\cref{mainquestion} has an affirmative answer with $c=1$ if $\GG$ has bounded degree, since every graph $G$ with treewidth less than $k$ and maximum degree $\Delta$ is contained in $T\boxtimes K_{18k\Delta}$ for some tree $T$ (with maximum degree at most $6\Delta$) \citep{DW22}. 

\cref{ReMainQuestion} presents several natural graph classes where there is an affirmative answer to \cref{mainquestion}, and several natural graph classes where \cref{mainquestion} is unsolved. 

\subsection{Bounded Treewidth Graphs}

Our final contribution concerns the product structure of graphs of bounded treewidth. Given that $H$ has bounded tree-depth in \cref{Almost}, it is natural to consider the following question: Given $k,n,\tdb\in\NN$ what is the minimum value of $m=m(k,n,\tdb)$ such that every $n$-vertex graph with treewidth $k$ is a subgraph of $H\boxtimes K_m$ for some graph $H$ with $\td(H)\leq\tdb$. We prove that
$m(k,n,\tdb)\in \Theta(n^{1/\tdb})$ for fixed $k$ and $\tdb$. The following theorem provides the upper bound.

\begin{restatable}{thm}{TreewidthTDproduct}
\label{TreewidthTDproduct}
For all $\tdb\in\NN$ and $k,n\in\NN$, every $n$-vertex graph $G$ with $\tw(G)\leq k$ is contained in $H\boxtimes K_m$ for some graph $H$ with $\td(H)\leq\tdb$, where $m\leq (k+1)^{1-1/\tdb}n^{1/\tdb}$.
\end{restatable}

The proof of \cref{TreewidthTDproduct} is based on a separator lemma for graphs of given treewidth that generalises \eqref{septw} and is of independent interest (\cref{TreewidthSep}). The lower bound, $m(k,n,\tdb)\in \Omega(n^{1/(\tdb+1)})$, follows by considering the case when $G$ is a path. Both these proofs are presented in \cref{BoundedTreewidthGraphs}.

Loosely speaking, \cref{TreewidthTDproduct} gives good bounds for graphs of bounded treewidth (or very small treewidth as a function of $n$), whereas \cref{Almost} (and the other results in \cref{BoundedTreedepth}) give better bounds in the strongly sublinear treewidth setting.

\section{Star Partitions}
\label{StarPartitions}

Recall \cref{Primary}, which shows that graph classes with strongly sublinear separation-number can be characterised via tree-partitions and star-partitions.

\Primary*

We now prove \cref{Primary}. Let (iv) be the statement that $\GG$ has strongly sublinear tree-width. It follows from the definitions that  $\tpw(G)\leq \spw(G)$ for every graph $G$; thus 
\cref{SSspw} $\Rightarrow$ \cref{SStpw}. \mycitet{Seese85} observed that $\tw(G) \leq 2 \tpw(G)-1$; thus \cref{SStpw} $\Rightarrow$ (iv). By \cref{Folklore}, (iv) $\Rightarrow$ \cref{SSS}. 

It remains to prove that \cref{SSS} $\Rightarrow$ \cref{SSspw}. Note that any graph $G$ has $\spw(G)\leq k$ if and only if $G$ has a set $S$ of at most $k$ vertices such that each component of $G-S$ has at most $k$ vertices. We thus use the following foklore result\footnote{\cref{LT} is proved by the following argument: initialise $S\coloneqq\emptyset$, while there is a component $X$ of $G-S$ with more than $n^\alpha$ vertices, add a balanced separator in $X$ to $S$, and repeat until every component of $G-S$ has at most $n^\alpha$ vertices. The idea is present in the work of \citet{LT79,LT80} for planar graphs, and in the work of \citet{EM94} and \citet{DN17} for hereditary graph class with strongly sublinear separators. See \citep{WoodSurvey} for an explicit proof of \cref{LT}.}.

\begin{lem}[\citep{EM94}] 
\label{LT} 
Let $\GG$ be any hereditary graph class with $\sep(\GG)\leq cn^{1-\epsilon}$, for some $c>0$ and $\epsilon\in(0,1)$. Then for any $\alpha\in(0,1)$ and any $n$-vertex graph $G\in \GG$ there exists $S\subseteq V(G)$ of size at most $\frac{c 2^\epsilon}{2^\epsilon-1} n^{1-\alpha\epsilon}$ such that each component of $G-S$ has at most $n^\alpha$ vertices.
\end{lem}

The next result follows from \cref{LT} with $\alpha=\frac{1}{1+\epsilon}$ (since  $1-\alpha\epsilon=1-\frac{\epsilon}{1+\epsilon}
=\frac{1}{1+\epsilon}$). 

\begin{cor}
\label{SepSPW}
Let $\GG$ be any hereditary graph class with $\sep(\GG)\leq cn^{1-\epsilon}$, for some $c>0$ and $\epsilon\in(0,1)$. Then for every $n$-vertex graph $G\in \GG$ there exists $S\subseteq V(G)$, such that $S$ and each component of $G-S$ has at most 
$\max\{\frac{c 2^\epsilon}{2^\epsilon-1},1\} n^{1/(1+\epsilon)}$ vertices, implying $$\spw(G) \leq 
\max\{\tfrac{c\, 2^\epsilon}{2^\epsilon-1},1\}\, n^{1/(1+\epsilon)}.$$
\end{cor}

\cref{SepSPW} shows that \cref{SSS} $\Rightarrow$ \cref{SSspw} in \cref{Primary}, which completes the proof of \cref{Primary}.

As a concrete example, if $\GG$ is any hereditary graph class with $\sep(\GG)\in O(n^{1/2})$, then $\spw(\GG)\in O(n^{2/3})$. For example, for every $n$-vertex planar graph $G$,
$$\tpw(G) \leq \spw(G) \in  O(n^{2/3}).$$
We now show this bound is tight. Consider a graph $G$ with $\tpw(G)\leq k$. A proper 2-colouring of the underlying tree determines an improper 2-colouring of $G$ such that each monochromatic component has at most $k$ vertices. \citet{LMST08} described an infinite class $\GG$ of planar graphs, such that every 2-colouring of any $n$-vertex graph in $\GG$ has a monochromatic component of order $\Omega(n^{2/3})$. So the $O(n^{2/3})$ upper bound on the tree-partition-width of planar graphs is tight.


To conclude this section, we show that the bound on the star-partition-width in \cref{SepSPW} is tight for grid graphs. For any integer $\dims\geq 2$, let \defn{$G_n^\dims$} be the $\dims$-dimensional $n^{1/\dims}\times\dots\times n^{1/\dims}$ grid graph (where $n^{1/\dims}\in\NN$). Our starting point is the following isoperimetric inequality by \citet[Theorem~3]{BL91}. For a graph $G$ and $A\subseteq V(G)$, let \defn{$\partial(A)$} be the number of edges in $G$ between $A$ and $V(G)\setminus A$. 

\begin{lem}[\citep{BL91}]
\label{Isoperimetric0}
For any set $A$ of vertices in $G_n^\dims$ with $|A|\leq \frac{n}{2}$, 
$$\partial(A) \geq 
\min\{ |A|^{1-1/r} r\,n^{1/r-1/\dims}: r \in\{1,\dots,\dims\} \}.$$ 
\end{lem}

\begin{cor}
\label{Isoperimetric}
For any set $A$ of vertices in $G_n^\dims$ with $|A|\leq \frac{n}{e^d}$, 
$$\partial(A) \geq \dims\, |A|^{1-1/\dims} .$$ 
\end{cor}

\begin{proof}
By \cref{Isoperimetric0} it suffices to show that if $|A|\leq \frac{n}{e^\dims}$ and $r\in\{1,\dots,\dims\}$, then 
\begin{equation}
   \label{Arnd}
    |A|^{1-1/r} rn^{1/r-1/\dims} \geq \dims|A|^{1-1/\dims}.
\end{equation}    
Since \eqref{Arnd} holds with equality when $r=\dims$, we may assume that $r\neq \dims$. 
Let $c\coloneqq\big(\frac{|A|}{n}\big)^{1/d}$ where $0\leq c\leq 1$, and let $x\coloneqq\frac{r}{\dims}$ where $0<x<1$. Then \eqref{Arnd} is equivalent to 
$x^{x/(1-x)} \geq c$. The function $x^{x/(1-x)}$ is decreasing when $x\in[0,1)$, and $\lim_{x\to1} x^{x/(1-x)}= \frac{1}{e}$. Thus \eqref{Arnd} holds if $c\leq \frac{1}{e}$; that is, if $|A|\leq \frac{n}{e^d}$.
\omitted{
$|A|^{1-1/r} rn^{1/r-1/\dims} \geq \dims|A|^{1-1/\dims}$\\
$(r/\dims)n^{1/r-1/\dims} \geq |A|^{1/r-1/\dims}$\\
$(r/\dims)n^{1/r-1/\dims} \geq (c^\dims n)^{1/r-1/\dims}$\\
$(r/\dims) \geq (c^\dims)^{1/r-1/\dims}$\\
$(r/\dims)^{r\dims} \geq (c^\dims)^{\dims-r}$\\
$(r/\dims)^{r} \geq c^{\dims-r}$\\
$x^{x\dims} \geq c^{(1-x)\dims}$\\
$x^{x/(1-x)} \geq c$}
\end{proof}

\begin{lem}
\label{CompGrid}
If $S$ is any set of at most $\frac{n}{2}$ vertices in $G_n^\dims$ and $q\leq \frac{n}{e^\dims}$ and each component of $G-S$ has at most $q$ vertices, then $|S|\geq\frac{n}{4q^{1/\dims}}$.
\end{lem}

\begin{proof}
Let $A_1,\dots,A_r$ be the components of $G^\dims_n-S$. 
By \cref{Isoperimetric} and since $G^\dims_n$ has maximum degree $2\dims$, 
$$  2\dims |S| \geq \sum_i \partial(A_i) 
\geq \sum_i \dims|A_i|^{1-1/\dims} 
\geq \sum_i \dims|A_i| q^{-1/\dims} 
= \dims q^{-1/\dims} (n-|S|)
\geq \tfrac12 (\dims q^{-1/\dims} n)
\enspace. $$
The result follows.
\end{proof}

Now consider the star-partition-width, $\spw(G_n^\dims)$. 
If $\spw(G_n^\dims)= s$ then $G_n^\dims$ has a set $S$ of at most $s$ vertices such that each component of $G_n^\dims-S$ has at most $s$ vertices.
If $s\le \tfrac{n}{e^\dims}$, then $s\geq \tfrac{n}{4s^{1/\dims}}$ by \cref{CompGrid}, and thus $s\geq(\frac{n}{4})^{\dims/(\dims+1)}$.
Hence $\spw(G_n^\dims)\geq(\frac{n}{4})^{\dims/(\dims+1)}$ when $n\ge \frac{e^{d(d+1)}}{4^d}$.

Let $\GG^\dims$ be the class of all subgraphs of $\dims$-dimensional grid graphs. Then $\sep(\GG^\dims)\leq cn^{1-1/\dims}$ for some $c=c(\dims)$ by a result of \citet{MTTV97}. Thus \cref{SepSPW} with $\epsilon=\frac{1}{\dims}$ proves 
$\spw(\GG^\dims) \leq 
\max\{\tfrac{c\, 2^{1/\dims}}{2^{1/\dims}-1},1\}\, 
n^{\dims/(\dims+1)}$, which matches the above lower bound. That is, for fixed $\dims$, $$\spw(\GG^\dims)=\Theta(n^{\dims/(\dims+1)}).$$


\section{Bounded Tree-depth Partitions}
\label{BoundedTreedepth}

This section shows that $n$-vertex graphs with strongly sublinear separation-number are contained in $H\boxtimes K_m$ where $\td(H)$ is bounded or is at most a slowly growing function of $n$, and $m$ is strongly sublinear with respect to $n$. All the results follow from the next lemma. 

\begin{lem}
\label{AlmostProof}
Let $\GG$ be any hereditary graph class with separation-number $\sep(\GG)\leq cn^{1-\epsilon}$, for some $c\geq 1$ and $\epsilon\in(0,1)$. Then for every $\tdb\in\NN$, every $n$-vertex graph $G\in\GG$ is contained in $H \boxtimes K_m$ for some graph $H$ with $\td(H)\leq\tdb$, where $m \leq \frac{c 2^\epsilon}{2^\epsilon-1}  n^{(1-\epsilon)/(1-\epsilon^\tdb)}$. 
\end{lem}

\begin{proof}
We proceed by induction on $\tdb$. 
If $\tdb=1$ then the claim is trivial---just take $H=K_1$.
Now assume that $\tdb\geq 2$ and the result holds for $\tdb-1$. 
Let $\alpha\coloneqq(1-\epsilon^{\tdb-1})/(1-\epsilon^\tdb)\in (0,1)$ and  $\gamma\coloneqq\frac{c 2^\epsilon}{2^\epsilon-1}$. 
Note that 
$$1-\alpha\epsilon
 = 
1-\frac{\epsilon(1-\epsilon^{\tdb-1})}{1-\epsilon^\tdb}
 = 
\frac{ 1-\epsilon^\tdb -\epsilon(1-\epsilon^{\tdb-1})}{1-\epsilon^\tdb}
 = 
\frac{ 1-\epsilon}{1-\epsilon^\tdb}.$$
Let $G$ be any $n$-vertex graph in $\GG$. By \cref{LT}, there exists $S\subseteq V(G)$ of size at most $$
 \gamma n^{1-\alpha\epsilon}
 = 
\gamma n^{( 1-\epsilon)/(1-\epsilon^\tdb)}$$ 
 such that each component of $G-S$ has at most $n^\alpha$ vertices. Say $G_1,\dots,G_k$ are the components of $G-S$. Let $n_i\coloneqq|V(G_i)| \leq n^\alpha$. By induction, $G_i$ is contained in $H_i \boxtimes K_q$ for some graph $H_i$ with $\td(H_i)\leq \tdb-1$, where 
 $$q \leq 
 \gamma n_i^{(1-\epsilon)/(1-\epsilon^{\tdb-1})}
\leq \gamma n^{\alpha(1-\epsilon)/(1-\epsilon^{\tdb-1})}
= \gamma n^{ (1-\epsilon)/(1-\epsilon^\tdb)}.$$
So $H_i$ is a subgraph of the closure of a rooted tree $T_i$ of vertex-height $\tdb-1$. Let $T$ be obtained from the disjoint union of $T_1,\dots,T_k$ by adding one root vertex $r$ adjacent to the roots of $T_1,\dots,T_k$. Let $H$ be the closure of $T$. So $\td(H)\leq\tdb$. By construction, $G$ is contained in $H\boxtimes K_m$ where 
$m \leq \max\{|S|,q \} \leq \gamma n^{(1-\epsilon)/(1-\epsilon^\tdb)}$, as desired. 
\end{proof}

Note that \cref{DimGridPartitionTreeDepthCor} in \cref{LowerBounds} shows that the $(1-\epsilon)/(1-\epsilon^\tdb)$ term in \cref{AlmostProof} is best possible whenever $\frac{1}{\epsilon}$ is an integer at least 2.

Recall \cref{Almost} from \cref{Introduction}.

\Almost*

\begin{proof}
Apply \cref{AlmostProof} with $\tdb \coloneqq \ceil{\log_\epsilon(\frac{\delta}{1-\epsilon+\delta})}$. Note that 
$\epsilon^\tdb(1-\epsilon+\delta)\leq \delta$, implying $1-\epsilon \leq (1-\epsilon)+\delta -\epsilon^\tdb(1-\epsilon+\delta)= (1-\epsilon+\delta)(1-\epsilon^\tdb)$. Thus every $n$-vertex graph $G\in\GG$ is contained in $H \boxtimes K_m$ for some graph $H$ with $\td(H)\leq\tdb$, where $m \leq 
\frac{c 2^\epsilon}{2^\epsilon-1} n^{(1-\epsilon)/(1-\epsilon^\tdb)}
= \frac{c 2^\epsilon}{2^\epsilon-1} n^{1-\epsilon+\delta}$.
\end{proof}

The next two results allow $\td(H)$ to increase slowly with $n$. 

\begin{thm}
\label{AlmostAlmost}
Fix $\epsilon\in(0,1)$. Let $\tdb:\mathbb{N}\to\mathbb{R}^+$ be any function. For any hereditary graph class $\GG$ with $\sep(\GG)\leq cn^{1-\epsilon}$ for some constant $c>0$, every $n$-vertex graph $G\in\GG$ is contained in $H \boxtimes K_m$ for some graph $H$ with $\td(H)\leq \ceil{\tdb(n)}$, where 
$m \leq \frac{c \, 2^\epsilon}{2^\epsilon-1}\, n^{1-\epsilon+\delta(n)}$ and $\delta(n)\coloneqq(1-\epsilon)
\big(
(1-\epsilon^{\tdb(n)})^{-1}
-1 \big)$.
\end{thm}

\begin{proof}
Apply \cref{AlmostProof} with $\tdb\coloneqq\ceil{\tdb(n)}$. Note that
$\delta(n) 
\geq (1-\epsilon)(\frac{1}{1-\epsilon^\tdb} -1) 
=
\frac{1-\epsilon}{1-\epsilon^\tdb} - ( 1-\epsilon) $, implying
$\frac{1-\epsilon}{1-\epsilon^\tdb} \leq 1-\epsilon+\delta(n)$.
Thus every $n$-vertex graph $G\in\GG$ is contained in $H \boxtimes K_m$ for some graph $H$ with $\td(H)\leq \ceil{\tdb(n)}$, where $m \leq  \frac{c \, 2^\epsilon}{2^\epsilon-1}\,n^{(1-\epsilon)/(1-\epsilon^\tdb)}\leq
 \frac{c \, 2^\epsilon}{2^\epsilon-1}\,n^{ 1-\epsilon+\delta(n)}$.
\end{proof}

The novelty of \cref{AlmostAlmost} is that $\tdb(n)$ can be chosen to be any slow-growing function, but the exponent-gap $\delta(n)$ goes to zero as $n\to\infty$. The next result with exponent-gap 0 follows from \cref{AlmostAlmost} by taking a specific function $h$.

\begin{thm}
\label{AlmostAlmostAlmost}
Fix $\epsilon\in(0,1)$ and $c>0$. For any hereditary graph class $\GG$ with $\sep(\GG)\leq cn^{1-\epsilon}$,  every $n$-vertex graph $G\in\GG$ is contained in $H \boxtimes K_m$ for some graph $H$ with $\td(H)\leq \bigl\lceil\frac{\log (1+\log n)}{-\log\epsilon}\bigr\rceil $, where $m \leq \frac{2c}{2^\epsilon-1}\, n^{1-\epsilon}$.
\end{thm}

\begin{proof}
Let $\tdb(n)\coloneqq \frac{\log (1+ \log n)}{-\log\epsilon} 
= \log_\epsilon( \frac{1}{1+\log n} ) $.
Thus
$\epsilon^{\tdb(n)} = \frac{1}{1+\log n}
= 1- \frac{\log n}{1+\log n}$, 
implying
$(1-\epsilon^{\tdb(n)})^{-1} -1 =
(\frac{1+\log n}{\log n}) -1 = 
\frac{1}{\log n}$. 
Hence
$\delta(n) \log n= (1-\epsilon)
\big(
(1-\epsilon^{\tdb(n)})^{-1}
-1 \big) \log n = 1-\epsilon$
and
$n^{\delta(n)} = 2^{1-\epsilon}$.
The result follows from 
\cref{AlmostAlmost} since
$m \leq 
\frac{c \, 2^\epsilon}{2^\epsilon-1}\, n^{1-\epsilon+\delta(n)}
=
\frac{c \, 2^\epsilon 2^{1-\epsilon}}{2^\epsilon-1}\, n^{1-\epsilon}
=
\frac{2c}{2^\epsilon-1}\, n^{1-\epsilon}
$.
\end{proof}

\section{Lower Bounds}
\label{LowerBounds}

This section proves lower bounds that show that several results in the previous section are best possible. Recall that $G_n^d$ is the $d$-dimensional $n^{1/d}\times\cdots\times n^{1/d}$ grid graph.

\begin{lem}
\label{DimPartialA}
Fix integers $\dims,\tdb\geq 2$. Let $G\coloneqq G^\dims_n$, let $H$ be any graph with $\td(H)\leq\tdb$, and let $A\subseteq V(G)$ such that $G[A]$ has an $H$-partition of width $s$  where 
\begin{equation}
\label{sBounds}
2^{\dims^{\tdb-1}} \leq 
s\leq 
\left(\tfrac{n}{36 e^\dims}\right)^{(\dims-1)\dims^{\tdb-2}/(\dims^{\tdb-1}-1)}.
\end{equation}
Then 
$$\partial(A)\geq 
\tfrac{\dims}{6} \, 
s^{-(\dims^{\tdb-1}-1)/((\dims-1)\dims^{\tdb-1})} |A| -3\dims s .$$
\end{lem}

\begin{proof}
We proceed by induction on $\tdb$ (with $\dims$ fixed). First suppose that $\tdb=2$. Let $R$ be the root part in the $H$-partition of $G[A]$. 
Thus $|R|\leq s$ and $\partial(A-R)\leq \partial(A) + \partial(R) \leq \partial(A) + 2\dims s$ (since $G$ has maximum degree at most $2\dims$). 
Let $C$ be the vertex-set of a component of $G[A]-R$.
Thus $|C|\leq s\leq\frac{n}{e^\dims}$ and 
$\partial(C)\geq \dims |C|^{1-1/\dims}$ by \cref{Isoperimetric}. Summing over all such $C$, 
$$ \partial(A-R) =
\sum_C \partial(C) \geq 
\sum_C \dims |C|^{1-1/\dims} \geq 
\sum_C \dims |C| s^{-1/\dims} = 
\dims s^{-1/\dims} (|A|-|R|) \geq 
\dims s^{-1/\dims} (|A|-s) .$$
Thus $$ \partial(A) 
\geq \partial(A-R) -2\dims s 
\geq  \dims s^{-1/\dims} (|A|-s) -2\dims s
> \tfrac{\dims}{6}\,s^{-1/\dims} |A| -3\dims s,$$
as claimed. 

Now assume that $\tdb\geq 3$ and the result holds for $\tdb-1$. 
Let $A\subseteq V(G)$ such that $G[A]$ has an $H$-partition of width $s$, for some graph $H$ with $\td(H)\leq\tdb$. 
Let $R$ be the root part in the $H$-partition of $G[A]$. 
Thus $|R|\leq s$ and 
\begin{equation}
\label{AR}
\partial(A-R)\leq \partial(A) + \partial(R) \leq \partial(A) + 2\dims s.
\end{equation}

Let $C$ be the vertex-set of a component of $G[A]-R$.
So $G[C]$ has an $H'$-partition of width $s$, for some graph $H'$ with $\td(H')\leq \tdb-1$.
Say $C$ is \defn{big} if 
$|C| \geq 
36 \, 
s^{
(\dims^{\tdb-1} -1 ) / ( (\dims-1)\dims^{\tdb-2} ) }
$ 
and \defn{small} otherwise. 

First suppose that $C$ is small.
Then $|C|\leq 36 \, 
s^{
(\dims^{\tdb-1} -1 ) / ( (\dims-1)\dims^{\tdb-2} ) }$, 
which is at most $\frac{n}{e^\dims}$ by the upper bound on $s$ in \cref{sBounds}. Thus \cref{Isoperimetric} is applicable, and 
\begin{align*}
\partial(C)\geq 
\dims|C|^{1-1/\dims}
  \geq
\tfrac{\dims}{6}\,|C| \,
 \big(
 s^{(\dims^{\tdb-1} -1)/((\dims-1)\dims^{\tdb-2})} \big)^{-1/\dims}
  = 
\tfrac{\dims}{6}\,
 s^{-(\dims^{\tdb-1}-1)/((\dims-1)\dims^{\tdb-1})} \,|C|.
\end{align*}

Now suppose that $C$ is big. 
Since \cref{sBounds} holds for $t$, it also holds for $t-1$. Thus, by induction, 
\begin{align*}
\partial(C)  \geq 
\tfrac{\dims}{6}\,
s^{-(\dims^{\tdb-2}-1)/((\dims-1)\dims^{\tdb-2})} |C| -3\dims s 
& \geq
\tfrac{\dims}{12}\,
s^{-(\dims^{\tdb-2}-1)/((\dims-1)\dims^{\tdb-2})} |C| \\
& \geq 
\tfrac{\dims}{6}\, 
s^{-(\dims^{\tdb-1}-1)/((\dims-1)\dims^{\tdb-1})}\,|C| ,
\end{align*}
where the final inequality holds since $s\geq 2^{\dims^{\tdb-1}}$.
\omitted{\begin{align*}
\tfrac{\dims}{12} \,
s^{(1-\dims^{\tdb-2})/(\dims-1)\dims^{\tdb-2}} |C|
&  \geq \tfrac{\dims}{6}\,
s^{( 1-\dims^{\tdb-1})/(\dims-1)\dims^{\tdb-1}}\,|C| \\
s^{(1-\dims^{\tdb-2})/(\dims-1)\dims^{\tdb-2} \,-\, 
( 1-\dims^{\tdb-1})/(\dims-1)\dims^{\tdb-1}} 
&  \geq 2\\
s^{(\dims-\dims^{\tdb-1} \,- 1+\dims^{\tdb-1}) / (\dims-1)\dims^{\tdb-1}} 
&  \geq 2\\
s^{1 / \dims^{\tdb-1}} 
&  \geq 2\\
s 
&  \geq 2^{\dims^{\tdb-1}}  
\end{align*}}
Summing over all such components, 
 \begin{align*}
 \partial(A-R) 
 =\sum_C \partial(C) 
 &  \geq \sum_{C} 
 \tfrac{\dims}{6} \,
s^{-(\dims^{\tdb-1}-1)/((\dims-1)\dims^{\tdb-1})} \,|C|\\
 & =
 \tfrac{\dims}{6} \,
s^{-(\dims^{\tdb-1}-1)/((\dims-1)\dims^{\tdb-1})}\,(|A|-|R|) \\
 & \geq
 \tfrac{\dims}{6} \,
s^{-(\dims^{\tdb-1}-1)/((\dims-1)\dims^{\tdb-1})} \,(|A|-s).
 \end{align*}
By \cref{AR},
\begin{align*}
\partial(A) 
\geq \partial(A-R)-2\dims s 
& \geq 
 \tfrac{\dims}{6} \,
s^{-(\dims^{\tdb-1}-1)/((\dims-1)\dims^{\tdb-1})} \,(|A|-s)
-2\dims s\\
& \geq 
 \tfrac{\dims}{6} \,
s^{(\dims^{\tdb-1}-1)/((\dims-1)\dims^{\tdb-1})} \,|A| 
-3\dims s,
\end{align*}
as desired. 
\end{proof}

\begin{lem}
\label{DimGridPartitionTreeDepth}
Fix integers $\dims,\tdb\geq 2$ and $s\geq 5^{\dims^\tdb}$. Let $G\coloneqq G^\dims_n$ where $n\gg \dims,\tdb$. Let $H$ be any graph with $\td(H)\leq\tdb$, such that $G$ has an $H$-partition of width $s$. Then 
$$s\geq \left(\tfrac{n}{12}\right)^{(\dims-1)\dims^{\tdb-1}/( \dims^{\tdb}-1)}. $$
\end{lem}

\begin{proof}
If $s\geq 
\left(\tfrac{n}{36 e^\dims}\right)^{(\dims-1)\dims^{\tdb-2}/(\dims^{\tdb-1}-1)}$ 
 then 
$s\geq \left(\tfrac{n}{12}\right)^{(\dims-1)\dims^{\tdb-1}/( \dims^{\tdb}-1)}$
for large enough $n\gg \dims,\tdb$,  
and we are done. 
Now assume that 
$s\leq 
\left(\tfrac{n}{36 e^\dims}\right)^{(\dims-1)\dims^{\tdb-2}/(\dims^{\tdb-1}-1)}$, which is required below when applying \cref{DimPartialA}. 

Let $M$ be the root part in the $H$-partition of $G$. Let $A_1,\dots,A_p$ be the vertex-sets of the components of $G-M$. 

First suppose that $\tdb=2$. Thus $|M|,|A_1|,\dots,|A_p|\leq s \leq\frac{n}{e^\dims}$. By \cref{CompGrid}, $s\geq |M|\geq\tfrac{n}{4s^{1/\dims}}$ and
$4s^{(\dims+1)/\dims} \geq n$ and
$$s 
\geq\left(\tfrac{n}{4}\right)^{\dims/(\dims+1)}
\geq\left(\tfrac{n}{12}\right)^{(\dims-1)\dims / (\dims^2-1)},$$
as desired. 

Now assume that $\tdb\geq 3$. 
By assumption, each $G[A_i]$ has an $H_i$-partition of width at most $s$, for some graph $H_i$ with $\td(H_i)\leq \tdb-1$. 
By \cref{DimPartialA}, 
$$ 
 \tfrac{\dims}{6} \,
 s^{-(\dims^{\tdb-2}-1)/((\dims-1)\dims^{\tdb-2})} |A_i| -3\dims s \leq 
\partial(A_i) \leq \partial(M)\leq 2\dims |M| \leq 2\dims s,$$
implying 
$$ |A_i| \leq q\coloneqq 
30 s^{1+ (\dims^{\tdb-2}-1)/((\dims-1)\dims^{\tdb-2})}
= 30 s^{ (\dims^{\tdb-1}-1)/((\dims-1)\dims^{\tdb-2})}.$$ 
Since $s\geq 5^{\dims^{\tdb}}$, 
$$\tfrac{30}{12} e^\dims \leq 
5^\dims \leq s^{1/\dims^{\tdb-1}} = 
s^{ (\dims^{\tdb}-1)/((\dims-1)\dims^{\tdb-1}) \,-\,  (\dims^{\tdb-1}-1)/((\dims-1)\dims^{\tdb-2})} ,$$
implying 
$$30 e^\dims \,
s^{ (\dims^{\tdb-1}-1)/((\dims-1)\dims^{\tdb-2})}
\leq 
12 \, s^{ (\dims^{\tdb}-1)/((\dims-1)\dims^{\tdb-1})} .$$
We may assume the right-hand-side is less than $n$, otherwise we are done. 
Thus
$$q=30 s^{ (\dims^{\tdb-1}-1)/((\dims-1)\dims^{\tdb-2})}
\leq \tfrac{n}{e^\dims} .$$
Hence \cref{Isoperimetric} is applicable to $A_i$, and 
\begin{align*}
2\dims s \geq 
\partial(M) = 
\sum_i\partial(A_i) \geq 
\sum_i \dims|A_i|^{1-1/\dims} \geq 
\sum_i \dims|A_i| q^{-1/\dims} 
& = \dims q^{-1/\dims} (n-|M|)\\
& \geq \dims q^{-1/\dims} (n-s) .
\end{align*}
Therefore
\begin{align*}
n  \leq 2sq^{1/\dims} + s
 = 2s\big(30 s^{(\dims^{\tdb-1}-1)/((\dims-1)\dims^{\tdb-2})}\big)^{1/\dims} + s
& = 2\cdot 30^{1/\dims}\, 
s^{ (\dims^{\tdb}-1)/((\dims-1)\dims^{\tdb-1})} + s\\
& < 12 s^{(\dims^{\tdb}-1)/((\dims-1)\dims^{\tdb-1})}.
 \end{align*}
 The result follows. 
\end{proof}

We now drop the $s\geq 5^{\dims^\tdb}$ assumption in \cref{DimGridPartitionTreeDepth}.

\begin{thm}
\label{DimGridPartitionTreeDepthCor}
Fix integers $\dims,\tdb\geq 2$. Let $G\coloneqq G^\dims_n$ where $n\gg \dims,\tdb$. For any graph $H$ with $\td(H)\leq\tdb$, if $G$ has an $H$-partition of width $s$, then 
$$s \geq  \left(\tfrac{n}{12}\right)^{(\dims-1)\dims^{\tdb-1} / (\dims^\tdb-1)}.
$$ 
\end{thm}

\begin{proof}
By \cref{DimGridPartitionTreeDepth} we may assume that $s<5^{\dims^{\tdb}}$.
Then $G$ has an $H$-partition of width $5^{\dims^{\tdb}}$.
By \cref{DimGridPartitionTreeDepth},
$$5^{\dims^{\tdb}} \geq \left(\tfrac{n}{12}\right)^{(\dims-1)\dims^{\tdb-1}/(\dims^{\tdb}-1)}.$$ 
Since $(\dims-1)\dims^{\tdb-1}/(\dims^{\tdb}-1)>0$, taking $n\gg \dims,\tdb$ we obtain a contradiction. \end{proof}

As mentioned earlier, subgraphs of $\dims$-dimensional grids are a hereditary class with separation-number $O(n^{1-1/\dims})$. The exponent of $n$ in the lower bound in \cref{DimGridPartitionTreeDepthCor} matches the corresponding upper bound in \cref{AlmostProof} with $\epsilon=\frac{1}{\dims}$ since 
$$(1-\epsilon)/(1-\epsilon^\tdb) 
= (1-\tfrac{1}{\dims})/(1-\tfrac{1}{\dims^t})
= \tfrac{\dims-1}{\dims}\cdot \tfrac{\dims^t}{\dims^\tdb-1}
= \tfrac{(\dims-1)\dims^{\tdb-1}}{\dims^\tdb-1}
.$$
Thus \cref{AlmostProof} is best possible whenever $\tfrac{1}{\epsilon}$  is an integer at least 2.

To conclude this section, we now show that the $O(\log\log n)$ term in 
\cref{AlmostAlmostAlmost} is best possible. 

\begin{thm}
\label{DimAlmostAlmostAlmostTight}
Fix any integer $\dims\geq 2$. Assume that there is a function $t$, such that for some $c>0$,  every $\dims$-dimensional grid graph  $G^\dims_n$ is contained in $H \boxtimes K_s$, for some graph $H$ with $\td(H)\leq t(n)$, and where $s \leq c n^{1-1/\dims}$. Then $t(n)\in \Omega(\log \log n)$.  
\end{thm}

\begin{proof}
Let $G\coloneqq G^\dims_n$ and $t\coloneqq t(n)$. By \cref{ProductPartition}, $G$ has an $H$-partition of width $s=\floor{c n^{1-1/\dims}}$. 
If $s\leq 5^{\dims^{t}}$ then 
$c n^{1-1/\dims} \leq 5^{\dims^{t}}+1$ 
and $t(n)\in \Omega(\log \log n)$, as desired. 
Otherwise, $s\geq 5^{\dims^{t}}$, and by \cref{DimGridPartitionTreeDepth},
$$cn^{1-1/\dims} \geq s \geq  \left(\tfrac{n}{12}\right)^{(\dims-1)\dims^{t-1}/(\dims^{t}-1)}.$$ 
Thus
$$
12c\geq  
c\, 12^{(\dims-1)\dims^{\tdb-1}/(\dims^t-1)}\geq  
n^{(\dims-1)\dims^{\tdb-1}/(\dims^{t}-1) - 1 + 1/\dims}
=
n^{ (\dims-1)/(\dims(\dims^{t}-1))}
.$$ 
Thus $(12c)^{\dims(\dims^t-1)/(\dims-1)}\geq   n$, and $t(n)\in \Omega(\log \log n)$, as desired. 
\end{proof}

Note that \cref{DimAlmostAlmostAlmostTight} implies that the strengthening of \cref{mainquestion} with tree-width replaced by tree-depth is false.

\section{Bounded Tree-width Graphs}
\label{BoundedTreewidthGraphs}

This section considers $H$-partitions of graphs with bounded tree-width, where $H$ has bounded tree-depth. We start with the most simple case: star-partitions of trees. 


\begin{lem}
\label{TreeSep}
For any $p,q,n\in\NN$ with $n\leq pq+p+q-1$, any $n$-vertex tree $T$ has a set $S$ of at most $p$ vertices such that each component of $T-S$ has at most $q$ vertices.
\end{lem}

\begin{proof} 
We proceed by induction on $p$. The base case with $p=1$ is folklore. We include the proof for completeness. Orient each edge $vw$ of $G$ from $v$ to $w$ if the component of $T-vw$ containing $v$ has at most $\frac{n}{2}$ vertices. Each edge is oriented by this rule. Since $T$ is acyclic, there is a vertex $v$ in $T$ with outdegree 0. So each component of $T-v$ has at most $\frac{n}{2}\leq q$ vertices, and the result holds with $S=\{v\}$. 

Now assume $p\geq 2$. Root $T$ at an arbitrary vertex $r$. For each vertex $v$, let $T_v$ be the subtree of $T$ rooted at $v$. For each leaf vertex $v$, let $f(v)=0$. For each non-leaf vertex $v$, let $f(v)\coloneqq \max_w |V(T_w)|$ where the maximum is taken over all children $w$ of $v$. If $f(r)\leq q$ then $S=\{r\}$ satisfies the claim. Now assume that $f(r)\geq q+1$. Let $v$ be a vertex of $T$ at maximum distance from $r$ such that $f(v)\geq q+1$. This is well-defined since $f(r)\geq q+1$. By definition, $|V(T_w)|\geq q+1$ for some child $w$ of $v$, but $f(w)\leq q$. So every subtree rooted at a child of $w$ has at most $q$ vertices. Let $T'$ be the subtree of $T$ obtained by deleting the subtree rooted at $w$. Thus $|V(T')|=n-|V(T_w)|\leq n-(q+1)\leq pq+p+q-1-(q+1)
=(p-1)q+(p-1)+q-1$. 
By induction, $T'$ has a set $S'$ of at most $p-1$ vertices such that each component of $T'-S'$ has at most $q$ vertices. Let $S\coloneqq S'\cup\{w\}$. By construction, $|S| \leq p$ and each component of $T-S$ has at most $q$ vertices.
\end{proof}

We now show the bound in \cref{TreeSep} is best possible for the $n$-vertex path $T$ with $n\not\equiv p \pmod{p+1}$. Say $T$ has a set $S$ of $p$ vertices such that each component of $T-S$ has at most $q$ vertices. Since $T$ is a path, $T-S$ has at most $p+1$ components, each with at most $q$ vertices. Thus $n\leq (p+1)q+p$. Since $n \not\equiv p \pmod{p+1}$, we have $n\leq (p+1)q+p-1$.

To generalise \cref{TreeSep} for graphs of given tree-width, we need the following normalisation lemma\footnote{Conditions (a), (b) and (c) in \cref{Normalisation} are well-known. Condition (d) may be new.}.

\begin{lem}
\label{Normalisation}
Every graph with tree-width $k$ has a tree-decomposition 
 $(B_x:x\in V(T))$ such that:
 \begin{enumerate}[\textnormal{(\alph{*})}]
\item $|B_x|=k+1$ for each $x\in V(T)$, 
 \item for each edge $xy\in E(T)$, we have $|B_x\setminus B_y|=1$ and $|B_y\setminus B_x|=1$,
\item $|V(T)|= |V(G)|-k$, and 
\item for any non-empty set $S\subseteq V(T)$, for each  component $T'$ of $T-S$, 
$$\Big|\big(\bigcup_{x\in V(T')}\!\!\!\!\!B_x\big)
\setminus
\big(\bigcup_{x\in S}B_x\big)\Big| 
\leq |V(T')|.$$
\end{enumerate}
\end{lem}

\begin{proof}
Since $G$ has tree-width $k$, $G$ has a tree-decomposition 
 $(B_x:x\in V(T))$ such that $|B_x|\leq k+1$ for each $x\in V(T)$. 
 
If $|B_x|>|B_y|$ for some edge $xy$ of $T$, then add one vertex from $B_x\setminus B_y$ to $B_y$, and repeat this step until $|B_x|=|B_y|$ for each edge $xy\in E(T)$. Now (a) is satisfied (since $G$ has tree-width $k$). 

If $B_x=B_y$ for some edge $xy\in E(T)$, then contract $xy$ into a new vertex $z$ with $B_z\coloneqq B_x$. This operation maintains that $|B_x|=k+1$ for each $x\in V(T)$. Repeat this operation until $B_x\neq B_y$ for each edge $xy\in E(T)$. 

Say $|B_x\setminus B_y|\geq 2$ for some edge $xy\in E(T)$. Thus $|B_y\setminus B_x|\geq 2$. Let $v\in B_x\setminus B_y$ and $w\in B_y\setminus B_x$. Delete the edge $xy$ from $T$, and introduce a new vertex $z$ to $T$ only adjacent to $x$ and $y$, where $B_z\coloneqq  (B_x\setminus \{v\})\cup\{w\}$. This operation maintains property (a) and still $B_x\neq B_y$ for each edge $xy\in E(T)$. Repeat this operation until (b) is satisfied.

We now prove (c). Let $r$ be any vertex of $T$. Define a function $f:V(T-r)\to V(G)\setminus B_r$, where for each $y\in V(T-r)$, if $p$ is the neighbour of $y$ in $T$ closest to $r$, then $f(y)$ is the vertex in $B_y\setminus B_p$. By (b), $f$ is a bijection. Thus $|V(G)\setminus B_r| = |V(T-r)|$ and $|V(G)|=|B_r|+|V(G)\setminus B_r|=(k+1)+(|V(T)|-1)=k+|V(T)|$. This proves (c). 

 We now prove (d). Let $S$ be a non-empty subset of $V(T)$. Let $T'$ be a component of $T-S$. Let $V' \coloneqq  (\bigcup_{y\in V(T')}B_y)\setminus(\bigcup_{x\in S}B_x)$. There is a vertex $x\in S$ adjacent to some vertex in $T'$. For each $y\in V(T')$, if $p$ is the neighbour of $y$ in $T$ closest to $x$, then let $v_y$ be the vertex in $B_y\setminus B_p$. Then $v_y\neq v_z$ for all distinct $y,z\in V(T')$. Each vertex of 
 $V'$ equals $v_y$ for some $y\in V(T')$. Thus $|V'| \leq |V(T')|$. This proves (d). 
  \end{proof}

\begin{thm}
\label{TreewidthSep}
Let $p,q,k,n\in\mathbb{N}$ with $n \leq \bigfloor{\frac{p}{k+1}}(q+1)+q+k-1$ and $p\geq k+1$. Then any $n$-vertex graph $G$ with $\tw(G)\leq k$ has a set $S$ of $p$ vertices such that each component of $G-S$ has at most $q$ vertices.
\end{thm}

\begin{proof}
Let $(B_x:x\in V(T))$ be a tree-decomposition of $G$ satisfying \cref{Normalisation}. Let $p'\coloneqq  \bigfloor{\frac{p}{k+1}}$. Thus $p'\geq 1$ and $p'(q+1) \geq n-q-k+1$, implying
$p'q+p'+q-1 \geq n-k = |V(T)|$ by \cref{Normalisation}(c). By \cref{TreeSep}, $T$ has a set $S_0$ of $p'$ vertices such that each component of $T-S_0$ has at most $q$ vertices. Let $S\coloneqq \bigcup_{x\in S_0}B_x$. So $|S|\leq (k+1)|S_0| \leq (k+1)p'\leq p$. For each component $G'$ of $G-S$ there is a subtree $T'$ of $T-S_0$ such that 
$V(G')\subseteq ( \bigcup_{x\in V(T')} B_x ) \setminus S$, implying $|V(G')|\leq |V(T')|\leq q$ by \cref{Normalisation}(d). \end{proof}

\cref{TreewidthSep} with $p=k+1$ and $q=\ceil{\frac{n-k}{2}}$ says that every graph $G$ with $\tw(G)\leq k$ has a set $S$ of $k+1$ vertices such that each component of $G-S$ has at most $\ceil{\frac{n-k}{2}}$ vertices, which implies \cref{septw} by \citet[(2.6)]{RS-II}. Thus \cref{TreewidthSep} generalises the result of \citet{RS-II} and is of independent interest.

We now show that \cref{TreewidthSep} is roughly best possible for the $k$-th
power of the path, $P_n^k$. This graph has vertex-set $\{v_1,\dots,v_n\}$ where
$v_iv_j$ is an edge if and only if $|i-j|\in\{1,\dots, k\}$. Note that
$\tw(P_n^k)=\pw(P_n^k)=k$.  Consider the graph $P_n^k$ for any integer $n>\tfrac{pq}{k}+p+q$.
Let $S$ be any set of $p$ vertices in $P_n^k$, and define a \defn{block}
to be a maximal set of vertices in $S$ that are consecutive with respect to $P_n$.
Let $b$ be the number of blocks of size at least $k$. So $p \geq bk$ and the
number of components of $P_n^k-S$ is at most $b+1$ (since if $B$ is a block of
size at most $k-1$, then the vertices immediately before and after $B$ are
adjacent in $P_n^k$).   Since $n>\tfrac{pq}{k}+p+q\ge (b+1)q+p$, it follows that
$P_n^k-S$ has a component with more than $q$ vertices.
Hence the bound in \cref{TreewidthSep} cannot be improved to $n\le \tfrac{pq}{k}+p+q+1$.

\cref{TreewidthSep} with $p=q=\ceil{\sqrt{(\tw(G)+1)n}}$ implies that for  every $n$-vertex graph $G$, 
$$\tpw(G) \leq \spw(G) \leq \ceil*{\sqrt{ (\tw(G)+1) n }}.$$ 
The second inequality here is essentially best possible since the argument above shows that $\text{spw}(P_n^k)=(1+o(1))\sqrt{kn}$. 
We now show that the same upper bound on $\tpw(G)$ is also essentially best possible. Let $k\geq 2$ and let $G$ be the graph obtained from $P_n^{k-1}$ by adding one dominant vertex. So $\tw(G)=k$. In any tree-partition of $G$, since $G$ has a dominant vertex, the tree is a star. Since $P_n^{k-1}$ is a subgraph of $G$, we have $\tpw(G)\geq (1-o(1))\sqrt{(k-1)n}$. 

We now set out to generalise \cref{TreewidthSep} for $H$-partitions, where $H$ has bounded tree-depth. We need the following analogue of \cref{TreeSep} where $q$ is allowed to be real.

\begin{lem}
\label{TreeSepNew}
For any $p,n\in\NN$ and $q\in\mathbb{R}^+$ with $n\leq q(p+1)$, any $n$-vertex tree $T$ has a set $S$ of at most $p$ vertices such that each component of $T-S$ has at most $q$ vertices.
\end{lem}

\begin{proof} 
We proceed by induction on $p$. The base case with $p=1$ is identical to the $p=1$ case of \cref{TreeSep}.  Now assume that $p\geq 2$. Root $T$ at an arbitrary vertex $r$. For each vertex $v$, let $T_v$ be the subtree of $T$ rooted at $v$. For each leaf vertex $v$, let $f(v)=0$. For each non-leaf vertex $v$, let $f(v)\coloneqq  \max_w |V(T_w)|$ where the maximum is taken over all children $w$ of $v$. If $f(r)\leq q$ then $S=\{r\}$ satisfies the claim. Now assume that $f(r)> q$. Let $v$ be a vertex of $T$ at maximum distance from $r$ such that $f(v)> q$. This is well-defined since $f(r)> q$. By definition, $|V(T_w)|> q$ for some child $w$ of $v$, but $f(w)\leq q$. So every subtree rooted at a child of $w$ has at most $q$ vertices. Let $T'$ be the subtree of $T$ obtained by deleting the subtree rooted at $w$. Thus $$|V(T')|=n-|V(T_w)|< n-q \leq q(p+1)-q=qp.$$ 
By induction, $T'$ has a set $S'$ of at most $p-1$ vertices such that each component of $T'-S'$ has at most $q$ vertices. Let $S\coloneqq S'\cup\{w\}$. By construction, $|S|\leq p$ and each component of $T-S$ has at most $q$ vertices.
\end{proof}

We have the following analogue of \cref{TreewidthSep} where $p$ and $q$ are both allowed to be real. 

\begin{lem}
\label{TreewidthSepNew}
Let $k,n\in\NN$ and $p,q\in\REALS_{>0}$ with $n(k+1) \leq pq$ and $p\geq k+1$. Then any $n$-vertex graph $G$ with $\tw(G)\leq k$ has a set $S$ of at most $p$ vertices such that each component of $G-S$ has at most $q$ vertices.
\end{lem}

\begin{proof}
Let $(B_x:x\in V(T))$ be a tree-decomposition of $G$ satisfying \cref{Normalisation}. Let $p'\coloneqq  \floor{\frac{p}{k+1}}$. Thus $p'\in\NN$ and $(p'+1)q \geq \frac{p}{k+1} q \geq  n$. By \cref{TreeSepNew}, $T$ has a set $S_0$ of at most $p'$ vertices such that each component of $T-S_0$ has at most $q$ vertices. Let $S\coloneqq \bigcup_{x\in S_0}B_x$. So $|S|\leq (k+1)|S_0| \leq (k+1)p'\leq p$. For each component $G'$ of $G-S$ there is a subtree $T'$ of $T-S_0$ such that $V(G')\subseteq ( \bigcup_{x\in V(T')} B_x ) \setminus S$, implying $|V(G')|\leq |V(T')|\leq q$ by \cref{Normalisation}(d). 
\end{proof}

We now reach the main result of this section (where the case $\tdb=2$ in \cref{TreewidthTDproduct} is roughly equivalent to \cref{TreewidthSep}). 

\TreewidthTDproduct*

\begin{proof}
We proceed by induction on $\tdb$. With $\tdb=1$, the claim holds with $H=K_1$ and $m=n$. Now assume $\tdb\geq 2$ and the claim holds for $\tdb-1$.
Without loss of generality, we can assume that $\tw(G)=k$, and thus $n\ge k+1$.
Let $p\coloneqq (k+1)^{1-1/\tdb}n^{1/\tdb}\ge k+1$ and $q\coloneqq 
(k+1)^{1/\tdb} n^{1-1/\tdb}$. So $n(k+1)=pq$. 
By \cref{TreewidthSepNew}, there is a set $S$ of at most $p$ vertices such that each component of $G-S$ has at most $q$ vertices. Let $G_1,\dots,G_c$ be the components of $G-S$. Each $G_i$ has treewidth at most $k$. By induction, $G_i$ is contained in $H_i\boxtimes K_{m'}$ for some graph $H_i$ with $\td(H_i)\leq \tdb-1$ and $m'\leq (k+1)^{(\tdb-2)/(\tdb-1)}q^{1/(\tdb-1)}=(k+1)^{1-1/\tdb}n^{1/\tdb}$. So $H_i$ is a subgraph of the closure of a rooted tree $T_i$ of vertex-height at most $\tdb-1$. Let $T$ be the rooted tree obtained from the disjoint union of $T_1,\dots,T_c$ by adding a root vertex $r$ adjacent to the root of each $T_i$. So $\td(H)\leq\tdb$. Let $S$ be the part associated with $r$. Hence $G$ is contained in $H\boxtimes K_m$ where $m\leq\max\{|S|,m'\}\leq (k+1)^{1-1/\tdb}n^{1/\tdb}$.
\end{proof}

We now show that the dependence on $n$ in \cref{TreewidthTDproduct} is best possible. In particular, we show that if the $n$-vertex path is contained in $H\boxtimes K_m$ for some graph $H$ with $\td(H)\leq\tdb$, then $m\geq \Omega(n^{1/\tdb})$ for fixed $\tdb$. We proceed by induction on $\tdb\geq 1$ with the following hypothesis: if the $n$-vertex path $P$ is contained in $H\boxtimes K_m$ for some graph $H$ with $\td(H)\leq\tdb$, then $n\leq (2m)^\tdb$. If $m=1$ then this says that $\td(P_n)\geq \log n$, which holds~\citep{Sparsity}. Now assume that $m\geq 2$. In the base case, if $\tdb=1$ then $H=K_1$ and $n\leq m$ as desired. Now assume that $\tdb\geq 2$, and the $n$-vertex path $P$ is contained in $H\boxtimes K_m$ for some graph $H$ with $\td(H)\leq\tdb$. Let $S$ be the set of at most $m$ vertices of $G$ associated with the root vertex $r$ of $H$. Since $P-S$ has at most $|S|+1$ components, some sub-path $P'$ of $P-S$ has $n'\geq \frac{n-|S|}{|S|+1} \geq\frac{n-m}{m+1}$ vertices.
If $\frac{n-m}{m+1} < \frac{n}{2m}$, then $n < \tfrac{2m^2}{m-1}<(2m)^2\le (2m)^\tdb$, as desired.
Hence, suppose that $n'\ge \frac{n-m}{m+1} \geq \frac{n}{2m}$.
Since $P'$ is connected, $P'$ is contained in $H'\boxtimes K_m$, where $H'$ is some component of $H-r$, which implies $\td(H')\leq \tdb-1$. By induction 
 $\frac{n}{2m} \leq n'\leq (2m)^{\tdb-1}$. Hence
 $n \leq (2m)^\tdb$, as desired.

\section{Regarding \cref{mainquestion}}
\label{ReMainQuestion}

This section presents several results related to \cref{mainquestion}. First we describe two methods, shallow minors and weighted separators, that can be used to show that various graph classes satisfy \cref{mainquestion}. We then give examples of graphs that highlight the difficulty of  \cref{mainquestion}. We conclude the paper by presenting several interesting graph classes for which \cref{mainquestion} is unsolved. 

\subsection{Using Shallow Minors}

For any integer $r\geq 0$, a graph $H$ is an \defn{$r$-shallow minor} of a graph $G$ if $H$ can be obtained from $G$ by contracting pairwise-disjoint subgraphs of $G$, each with radius at most $r$, and then taking a subgraph. Let $\nabla_r(G)$ be the maximum average degree of an $r$-shallow minor of $G$. Shallow minors are helpful for attacking \cref{mainquestion}. \citet{HW21b} proved the following, where $\Delta(G)$ is the maximum degree of $G$. 

\begin{lem}[\citep{HW21b}]
\label{ShallowMinorGeneral}
For any $r\in\mathbb{N}_0$ and $\ell,t\in\mathbb{N}$, 
for any graphs $H$ and $L$ where $\tw(H)\leq t$ and $\Delta( L^r )\leq k$, 
if a graph $G$ is an $r$-shallow minor of $H \boxtimes L\boxtimes K_\ell$, then $G$ is contained in $J\boxtimes   L^{2r+1} \boxtimes   K_{\ell( k+1)}$ for some graph $J$ with $\tw(J)\leq \binom{2r+1+t}{t}-1$. 
\end{lem}

\cref{ShallowMinorGeneral} with $L=K_1$ and $k=0$ implies:

\begin{cor}
\label{ShallowMinor}
For any graph $H$ with $\tw(H)\leq t$, for any $r\in\mathbb{N}_0$ and $\ell\in\mathbb{N}$, if a graph $G$ is an $r$-shallow minor of $H \boxtimes K_\ell$, then $G$ is contained in $J\boxtimes K_{\ell}$ for some graph $J$ with $\tw(J) \leq \binom{2r+1+t}{t}-1$.
\end{cor}

\cref{ShallowMinor} essentially says that shallow minors of graphs that satisfy \cref{mainquestion} also satisfy \cref{mainquestion}. We now give two examples of this approach.

A graph $G$ is \defn{$(g,k)$-planar} if there is a drawing of $G$ in a surface of Euler genus $g$ with at most $k$ crossings on each edge (assuming no three edges cross at a single point). 
\begin{prop}
Every $n$-vertex $(g,k)$-planar graph $G$ is contained in $L\boxtimes K_{\ell}$ for some graph $L$ with $\tw(L) \leq \binom{k+5}{3}-1$, where $\ell
\leq 8\sqrt{(g+1)(1+k^{3/2}d_g) n}$ and $d_g\coloneqq \max\{3,\frac14( 5+\sqrt{24g+1})\}$. 
\end{prop}

\begin{proof}
\cref{ISW}(b) establishes the $k=0$ case. Now assume that $k\geq 1$. Fix a drawing of $G$ in a surface of Euler genus $g$ with at most $k$ crossings on each edge. Let $C$ be the total number of crossings. Let $m\coloneqq |E(G)|$. So $C \leq \frac{k}{2}m$. 
\citet[Lemma~4.5]{OOW19} proved the following generalisation of the Crossing Lemma: if $m>2d_g n$ then $C\geq \frac{m^3}{8 (d_gn)^2}$. Assume for the time being that $m>2d_g n$.
Thus
$ \frac{m^3}{8 (d_gn)^2} \leq C \leq \frac{k}{2}m$, implying
$ m^2 \leq 4k (d_gn)^2$ and
$ m \leq 2\sqrt{k} d_gn$.
Thus $m \leq \max\{ 2d_g ,2\sqrt{k} d_g \} n = 2\sqrt{k} d_gn$.
Hence $C \leq \frac{k}{2}m \leq k^{3/2} d_g  n$. 
\citet{HW21b} showed that $G$ is a $\ceil{\frac{k}{2}}$-shallow minor of $H\boxtimes K_2$, where $H$ is the graph of Euler genus $g$ obtained from $G$ by adding a vertex at each crossing point. Hence $|V(H)|\leq n+C <  (1+k^{3/2}d_g) n$. By \cref{ISW}(b), $H$ is contained in $J\boxtimes K_{\ell'}$, for some graph $J$ with $\tw(J)\leq 3$, where $\ell'\leq 4\sqrt{(g+1)|V(H)|} \leq 4\sqrt{(g+1)(1+k^{3/2}d_g) n}$. Hence $G$ is a $\ceil{\frac{k}{2}}$-shallow minor of $J\boxtimes K_{2\ell'}$. By \cref{ShallowMinor} with $t=3$ and $r=\ceil{\frac{k}{2}}$, we have that $G$ is contained in $L\boxtimes K_{2\ell'}$ for some graph $L$ with $\tw(L) \leq \binom{k+5}{3}-1$.
\end{proof}

Here is a second  example. A graph $G$ is \defn{fan-planar} if there is a drawing of $G$ in the plane such that for each edge $e \in E(G)$ the edges that cross $e$ have a common end-vertex and they cross $e$ from the same side (when directed away from their common end-vertex) \citep{BG20,KU22}. 

\begin{cor}
Every $n$-vertex fan-planar graph $G$ is contained in $J\boxtimes K_m$ where $\tw(J)\leq 19$ and $m<29\sqrt{n}$. 
\end{cor}

\begin{proof}
\citet{KU22} proved that $G$ has less than $5n$ edges. 
\citet{HW21b} showed that $G$ is a 1-shallow minor of $H\boxtimes K_3$ for some planar graph $H$ with $|V(H)|\leq |V(G)|+2|E(G)|<11n$. By \cref{ISW}(b), $H$ is contained in $J\boxtimes K_{m'}$, for some graph $J$ with $\tw(J)\leq 3$, where $m'\leq 2\sqrt{2|V(H)|} \leq 2\sqrt{22n}<\tfrac{29}{3}\sqrt{n}$. Thus $G$ is a 1-shallow minor of $J\boxtimes K_{3m'}$. By \cref{ShallowMinor} with $t=3$, $G$ is contained in $J\boxtimes K_{3m'}$ for some graph $J$ with $\tw(J) \leq \binom{6}{3}-1=19$.
\end{proof}

\subsection{Using Weighted Separators}

The following definition and lemma provides another way to show that various graph classes satisfy \cref{mainquestion}. A graph $J$ is \defn{$(n,m)$-separable} if for every vertex-weighting of $J$ with non-negative real-valued weights and with total weight $n$, there is a set $S\subseteq V(J)$ with total weight $m$, such that each component of $J-S$ has at most $m$ vertices. There are numerous results about weighted separators in the literature, but most of these consider the total weight of each component of $J-S$ instead of the weight of $S$ itself. Such considerations are studied in depth by \citet{Dvorak22}.

\begin{lem}
\label{Separable}
For any graph $H$ and any $(n,m)$-separable graph $J$, if $G$ is any $n$-vertex graph contained in $H\boxtimes J$, then $G$ is contained in $L\boxtimes K_m$ for some graph $L$ with $\tw(L)\leq\tw(H)+1$. 
\end{lem}

\begin{proof}
We may assume that $G$ is a subgraph of $H\boxtimes J$. So each vertex of $G$ is of the form $(x,y)$ where $x\in V(H)$ and $y\in V(J)$. Weight each vertex $y$ of $J$ by the number of vertices $x\in V(H)$ such that $(x,y)$ is in $G$. So the total weight is $n$. By assumption, there is a set $S\subseteq V(J)$ with total weight $m$, where each component of $J-S$ has at most $m$ vertices. Let $J_1,\dots,J_t$ be the components of $J-S$. Let $H_1,\dots,H_{t}$ be disjoint copies of $H$. Let $L$ be obtained from $H_1\cup\dots\cup H_{t}$ by adding one dominant vertex $z$. Note that $\tw(L) \leq \tw(H)+1$. For each vertex $x$ of $H$, let $x_i$ be the copy of $x$ in $H_i$. 

We now define a partition of $V(G)$ indexed by $V(L)$. Let $V_z\coloneqq \{(x,y)\in V(G):x\in V(H),y\in S\}$. So $|V_z|= \text{weight}(S)\leq m$. For each vertex $x_i$ of $L-z$, let $V_{x_i}\coloneqq \{(x,y)\in V(G): y \in V(J_i)\}$. So $|V_{x_i}| \leq |V(J_i)|\leq m$. Let $\PP\coloneqq \{V_z\}\cup\{V_{x_i}: x\in V(L),i\in\{1,\dots,t\}\}$.
By construction, $\PP$ is a partition of $V(G)$, and each part of $\PP$ has size at most $m$. 

We now verify that $\PP$ is an $L$-partition of $G$. Consider an edge $vv'$ of $G$. The goal is to show that $vv'$ `maps' to a vertex or edge of $L$. If $v\in V_z$ or $v'\in V_z$ then $vv'$ maps to a vertex or edge of $L$ (since $z$ is dominant in $L$). Otherwise,  $v\in V_{x_i}$ and $v'\in V_{x'_j}$ for some $x,x'\in V(H)$. By the definition of $V_{x_i}$, we have $v=(x,y)$ for some $y\in V(J_i)$. Similarly, $v'=(x',y')$ for some $y'\in V(J_j)$. Since $vv'\in E(G)$, we have $x=x'$ or $xx'\in E(H)$, and $y=y'$ or $yy'\in E(J-S)$. If $yy'\in E(J-S)$, then $y$ and $y'$ are in the same component of $J-S$, implying $i=j$. If $y=y'$ then $i=j$ as well. In both cases, $x_i=x'_j$ or $x_ix'_j\in E(L)$. Hence, $vv'$ maps to a vertex or edge of $L$. 

This shows that $\PP$ is an $L$-partition of $G$ with width at most $m$. By \cref{ProductPartition}, $G$ is contained in $L\boxtimes K_{m}$.
\end{proof}

Note that \cref{Separable} holds even when $J$ is only $(n,m)$-separable with integer-valued weights.
Our goal now is to find graphs that are $(n,m)$-separable where $m\in O(n^{1-\epsilon})$. 

\begin{lem}
\label{PathKcSeparable}
For every path $P$ and $c,n\in\NN$, the graph $G\coloneqq  P\boxtimes K_c$ is $(n,\sqrt{cn})$-separable. 
\end{lem}

\begin{proof}
Say $P=(v_1,v_2,\dots)$ and $V(K_c)=\{1,\dots,c\}$. 
Assume the vertices of $G$ are assigned non-negative weights, with total weight $n$. 
Let $m\coloneqq \ceil{\sqrt{n/c}}$. 
For $i\in\{1,\dots,m\}$, let $S_i\coloneqq  \{(v_j,\ell):j\equiv i\pmod{m},\ell\in\{1,\dots,c\}\}$. So $S_1,\dots,S_m$ is   partition of $V(G)$. Thus $S_{i^\star}$ has total weight at most $\frac{n}{m}$ for some $i^\star\in\{1,\dots,m\}$. Each component of $G-S_{i^\star}$ has at most $c(m-1)$ vertices. The result follows since 
$\frac{n}{m}
\leq
\sqrt{cn}$ and 
$c(m-1) < \sqrt{cn}$.
\end{proof}

\cref{Separable,PathKcSeparable} together imply:

\begin{lem}
\label{RowTreewidthGen}
For any graph $H$, path $P$ and $c\in\NN$, if $G$ is any $n$-vertex graph contained in $H\boxtimes P\boxtimes K_c$, then $G$ is contained in $L\boxtimes K_m$ for some graph $L$ with $\tw(L)\leq\tw(H)+1$, where $m\leq\sqrt{cn}$. 
\end{lem}

Several recent results show that certain graphs $G$ are contained in $H\boxtimes P\boxtimes K_c$, for some path $P$ and graph $H$ with bounded treewidth~\citep{DJMMUW20,UWY22,DHHW22,ISW,DMW23,HW21b}. In all these cases, \cref{RowTreewidthGen} is applicable, implying that $G$ is contained in $L\boxtimes K_{O(\sqrt{n})}$, for some graph $L$ with bounded tree-width. 

We give one example: map graphs. Start with a graph $G$ embedded without crossings in a surface of Euler genus $g$, with each face labelled a `nation' or a `lake', where each vertex of $G$ is incident with at most $d$ nations. Let $M$ be the graph whose vertices are the nations of $G$, where two vertices are adjacent in $G$ if the corresponding faces in $G$ share a vertex. Then $M$ is called a \defn{$(g,d)$-map graph}. Since the graphs of Euler genus $g$ are precisely the $(g,3)$-map graphs~\citep{DEW17}, map graphs are a natural generalisation of graphs embeddable in surfaces. \citet{DHHW22} proved that any $(g,d)$-map graph is contained in  $H\boxtimes P\boxtimes K_\ell$ for some planar graph $H$ with treewidth 3 and for some path $P$, where $\ell=\max\{2g\floor{\frac{d}{2}},d+3\floor{\frac{d}{2}}-3\}$. The next result thus follows from \cref{RowTreewidthGen}.

\begin{prop}
\label{MapGraphs}
Every $n$-vertex $(g,d)$-map graph $G$ is contained in $H\boxtimes K_m$ for some apex graph $H$ with $\tw(H)\leq 4$, where $m\leq\sqrt{\ell n}$ and $\ell\coloneqq \max\{2g\floor{\frac{d}{2}},d+3\floor{\frac{d}{2}}-3\}$.
\end{prop}

\citet{DEW17} showed that $n$-vertex $(g,d)$-map graphs have separation-number $\Theta(\sqrt{(g+1)(d+1)n})$. Thus \cref{MapGraphs} gives another example of a graph class with separation-number $cn^{1-\epsilon}$, where $\tw(H)$ is independent of $c$ in the corresponding product structure theorem. (Compare with the discussion after \cref{mainquestion}.)\ 

Motivated by \cref{Separable}, we give three more examples of $(n,m)$-separable graphs. The first is a multi-dimensional generalisation of \cref{PathKcSeparable}.

\begin{prop}
\label{PathProductKcSeparable}
For all paths $P_1,\dots,P_d$ and $c,n\in\NN$, the graph $G\coloneqq  P_1\boxtimes \dots\boxtimes P_d\boxtimes K_c$ is $(n,(dn)^{d/(d+1)} c^{1/(d+1)})$-separable. 
\end{prop}

\begin{proof}
Say $P_i=(1,2,\dots)$ and $V(K_c)=\{1,\dots,c\}$. Assume the vertices of $G$ are assigned non-negative weights, with total weight $n$. 
Let $m\coloneqq \ceil{(dn/c)^{1/(d+1)}}$.  
For $j\in\{1,\dots,m\}$, let 
$S_j$ be the set of all vertices $(x_1,\dots,x_d,\ell)$ in $G$ where $x_i\equiv j\pmod{m}$ for some $i\in\{1,\dots,d\}$. 
Each vertex of $G$ is in at least one and in at most $d$ such sets. Thus, the total weight of $S_1,\dots,S_m$ is at most $dn$. Thus $S_{j^\star}$ has total weight at most $\frac{dn}{m}$ for some $j^\star\in\{1,\dots,m\}$. Each component of $G-S_{j^\star}$ has at most $c(m-1)^d$ vertices. The result follows since 
$\frac{dn}{m}
\leq dn / (dn/c)^{1/(d+1)}
= (dn)^{d/(d+1)} c^{1/(d+1)}
$ and 
$c(m-1)^d 
\leq c (dn/c)^{d/(d+1)}
= (dn)^{d/(d+1)} c^{1/(d+1)}$
\end{proof}

Now consider trees. 

\begin{prop}
\label{TreeSeparable}
Every $n$-vertex tree $T$ with maximum degree $\Delta\geq 3$ is $\left(n,\frac{(1+o(1))n}{\log_{\Delta-1}n}\right)$-separable. 
\end{prop}

\begin{proof}
Assume the vertices of $T$ are assigned non-negative weights, with total weight $n$. In this proof all logs are base $\Delta-1$.  
Let $m\coloneqq \ceil{\log n - \log \log n}\geq 2$. 
Root $T$ at a leaf vertex $r$. So each vertex has at most $\Delta-1$ children. For $j\in\{1,\dots,m\}$, let $V_j\coloneqq \{v\in V(T):\dist_T(v,r) \equiv j \pmod{m}\}$. 
Thus $V_1,\dots,V_m$ is a partition of $V(T)$. 
There exists $j^\star$ such that $V_{j^\star}$ has weight at most $\frac{n}{m}$. 
In $T-V_{j^\star}$, each component has radius at most $m-2$, so the number of vertices is at most 
$\sum_{i=0}^{m-2}(\Delta-1)^i
= ( (\Delta-1)^{m-1} -1)/(\Delta-2) <(\Delta-1)^{m-1}$. 
The result follows since 
$\frac{n}{m}\leq \frac{n}{\log n - \log \log n}\leq\frac{(1+o(1))n}{\log n}$
and
$(\Delta-1)^{m-1} < (\Delta-1)^{\log n - \log \log n} = \frac{n}{\log n}$.
\end{proof}

To get the $(n,o(n))$-separable result in \cref{TreeSeparable}, the bounded degree assumption is necessary. Suppose that for all $n$ every star $T$ is $(n,m)$-separable where $m<\frac{n}{3}$. Let $T$ be the star with $p$ leaves. Let $r$ be the centre vertex of $T$. Assign each leaf a weight of 1 and assign $r$ a weight of $\frac{p}{2}$. The total weight is $n\coloneqq \frac{3p}{2}$. So for some $m<\frac{n}{3}=\frac{p}{2}$ there is a set $S$ of 
vertices in $T$ with weight at most $m$ such that each component of $T-S$ has at most $m$ vertices. Since $r$ has weight $\frac{p}{2} > m$, we have $r\not\in S$. So every vertex in $S$ is a leaf, each of which has weight 1. So $|S|\leq m$. Thus $T-S$ is a star with at least $p-m> m$ leaves, which is a contradiction. Hence $m\geq\frac{n}{3}$. 

\cref{TreeSeparable} generalises as follows. 

\begin{prop}
\label{TreewidthDegreeSeparable}
For all $\Delta,k\in\NN$ there exists $\alpha>0$ such that every graph $G$ with maximum degree $\Delta$ and treewidth $k$ is $\left(n,\frac{\alpha n}{\log n}\right)$-separable. 
\end{prop}

\begin{proof}
Let $c\coloneqq 18(k+1)\Delta$. \citet{DW22} proved that $G$ is contained in $T\boxtimes K_c$ for some tree  $T$ with maximum degree at most $6\Delta$. Observe that if $H$ is any $(n,m)$-separable graph and $c\in\NN$, then $H\boxtimes K_c$ is $(n,cm)$-separable. 
By \cref{TreeSeparable}, $T\boxtimes K_c$ is $\left(n,\frac{c(1+o(1))n}{\log_{6\Delta-1}n}\right)$-separable. 
Observe that if $H$ is any $(n,m)$-separable graph, then every subgraph of $H$ is $(n,m)$-separable. Thus $G$ is 
$\left(n,\frac{c(1+o(1))n}{\log_{6\Delta-1}n}\right)$-separable. \end{proof}

\subsection{Bad News}

We now present a result that highlights the difficulty of \cref{mainquestion}. For simplicity, we focus on the $\epsilon=\frac12$ case. \cref{BadNews} below shows there are $n$-vertex graphs $G$ with $\td(G)\leq\sqrt{n}$ such that $G$ is contained in no graph $H\boxtimes K_m$ with $m\in O(\sqrt{n})$ and $\tw(H)$ bounded. 

\begin{prop}
\label{BadNews}
For any $c\in\NN$ there exist infinitely many $n\in\NN$ for which there is an $n$-vertex graph $G$ with $\td(G)\leq\sqrt{n}$ such that for any graph $H$, if $G$ is contained in $H\boxtimes K_m$ with  $m\leq c\sqrt{n}$, then 
$\omega(H)\geq \frac{ \log n}{4\log(c\log n)}$.
\end{prop}

\begin{proof}
Fix any integer $\ell\geq 2$. 
Let 
$d\coloneqq c\ell^2$ and 
$h\coloneqq  c^{\ell-1} \ell^{2\ell-4}$. 
Note that $h,\ell\in\NN$.

For $j\in\{1,\dots,\ell\}$, let $T_j$ be the complete $d$-ary tree of vertex-height $j$, and let $T'_j$ be the $(h-1)$-subdivision of $T_j$. Consider $T'_1 \subseteq T'_2 \subseteq \dots\subseteq T'_\ell$  with a common root vertex $r$. 

\begin{claim}
For every partition $\PP$ of $V(T'_j)$ with parts of size at most $\frac{hd-1}{\ell-1}$ there exists a root--leaf path in $T'_j$ intersecting at least $j$ parts of $\PP$. 
\end{claim}

\begin{proof}
We proceed by induction on $j$. Since $|V(T'_1)|=1$ the $j=1$ case is trivial. Assume the $j-1$ case holds. Let $\PP$ be any partition of $V(T'_j)$ with parts of size at most $\frac{hd-1}{\ell-1}$.
By induction, there is a leaf $v$ of $T'_{j-1}$ such that the $vr$-path in $T'_{j-1}$ intersects at least $j-1$ parts $P_1,\dots,P_{j-1}$ of $\PP$. Let $X$ be the set of descendants of $v$. If $X \subseteq P_1\cup \dots \cup P_{j-1}$ then $$hd=|X|\leq |P_1\cup\dots\cup P_{j-1}|\leq (j-1)\,\tfrac{hd-1}{\ell-1} \leq hd-1,$$
which is a contradiction. Thus there is a vertex $x\in X\setminus (P_1\cup\dots\cup P_{j-1})$. Let $y$ be any leaf-descendent of $x$. Thus, the $yr$-path in $T'_j$ 
intersects at least $j$ parts of $\PP$, as desired. \end{proof}

Let $G$ be the closure of $T'_\ell$. 
Let 
$$n\coloneqq |V(G)| = h (\tfrac{d}{d-1})( d^{\ell-1}-1)+1.$$ 
We need the following lower bound on $n$:
\begin{align}
n \geq h d^{\ell-1} 
= c^{\ell-1} \ell^{2\ell-4} (c\ell^2)^{\ell-1} 
= c^{2\ell-2} \ell^{4\ell-6}. \label{nlb}
\end{align}
And we need the following upper bound on $n$:
\begin{align}
n(\ell-1)^2 & =
( h (\tfrac{d}{d-1})( d^{\ell-1}-1)+1 )(\ell-1)^2 \nonumber\\
& \leq h (\tfrac{d}{d-1})( d^{\ell-1}) (\ell-1)^2 \nonumber\\
& < h  d^{\ell-1} \ell^2 \nonumber\\
& =  (c^{\ell-1} \ell^{2\ell-4}) d^2 (c\ell^2)^{\ell-3} \ell^2 \nonumber\\
& =  c^{2\ell-4} d^2 \ell^{4\ell-8} \label{nub}.
\end{align}
The vertex-height of $T'_\ell$ equals $h(\ell-1)+1$. By \eqref{nlb}, 
$$\td(G) = h(\ell-1)+1 
\leq h\ell = c^{\ell-1} \ell^{2\ell-3} \leq \sqrt{n}.$$
Assume that $G$ is contained in $H\boxtimes K_m$ for some graph $H$ and integer $m\leq c\sqrt{n}$. By \eqref{nub}, 
$$m(\ell-1) 
\leq c\sqrt{n}(\ell-1) 
< c^{\ell-1} \ell^{2\ell-4} d 
= hd.$$
Thus $m\leq\frac{hd-1}{\ell-1}$. By \cref{ProductPartition}, there is an $H$-partition of $G$ with width at most $\frac{hd-1}{\ell-1}$. 
By the claim, there exists a root--leaf path in $T'_\ell$ intersecting at least $\ell$ parts of $\PP$. 
These $\ell$ parts form a clique in $H$. 
Thus $\omega(H) \geq \ell $.
By \cref{nub},  $n < c^{2\ell} \ell^{4\ell}$, which implies
$\omega(H) \geq \ell \geq 
\frac{\log n}{4 \log(c\log n)}$, as desired. 
\end{proof}

\cref{BadNews} is not a negative answer to \cref{mainquestion} since $G$ is a single graph, not a hereditary class. Indeed, the graphs in \cref{BadNews} have unbounded complete subgraphs,  and therefore are in no hereditary class with strongly sublinear separation-number. This result can be interpreted as follows: A natural strengthening of \cref{mainquestion} (with $\epsilon=\frac12$) says that every $n$-vertex graph $G$ with $\td(G)\in O(\sqrt{n})$ is contained in $H\boxtimes K_m$, for some graph $H$ with $O(1)$ treewidth, where $m\in  O(\sqrt{n})$. \cref{BadNews} says this strengthening is false. So it is essential that $\GG$ is a hereditary class in \cref{mainquestion}.

\subsection{Future Directions}

We conclude by listing several unsolved special cases of \cref{mainquestion} of particular interest:
\begin{itemize}
\item Does \cref{mainquestion} hold for touching graphs of 3-D spheres, which have  $O(n^{2/3})$ separation-number  \citep{MTTV97}?

\item \citet{EG17} defined a graph $G$ to be \defn{$k$-crossing-degenerate} if $G$ has a drawing in the plane such that the associated crossing graph is $k$-degenerate. They showed that such graphs have $O(k^{3/4}n^{1/2})$ separation-number. It is open whether \cref{mainquestion} holds for $k$-crossing-degenerate graphs. The same question applies for $k$-gap-planar graphs~\citep{GapPlanar18}. 
\item Does \cref{mainquestion}  hold for string graphs on $m$ edges, which have  $O(m^{1/2})$ separation-number  \citep{Lee16,Lee17}? 
\item Does \cref{mainquestion} hold for graphs with layered tree-width $k$, which have $O(\sqrt{kn})$ separation-number \citep{DMW17}?
\item See \citep{SW98,DMN21} for many geometric intersection graphs where \cref{mainquestion} is unsolved and interesting. 
\end{itemize}

\subsection*{Acknowledgement}

This work was initiated at the 2022 Workshop on Graphs and Geometry held at the Bellairs Research Institute. Thanks to the organisers for creating a wonderful working environment, and to  Robert Hickingbotham  for helpful conversations. Thanks to Daniel Harvey who corrected an error in the  first proof of \cref{TreeSep}.

\small
\bibliographystyle{DavidNatbibStyle}
\bibliography{DavidBibliography}
\end{document}